\documentclass[12pt,a4paper,reqno]{amsart}
\usepackage{amssymb}
\usepackage{latexsym}
\usepackage{amsmath}
\usepackage{mathrsfs}
\usepackage{cite}

\usepackage{calc}

\newcommand{\scal}[2]{\langle #1,#2\rangle}
\newcommand{\rr}[1]{\mathbf R^{#1}}

\newcommand{\nm}[2]{\Vert #1\Vert _{#2}}

\newcommand{\op}{\operatorname{Op}}
\newcommand{\sets}[2]{\{ \, #1\, ;\, #2\, \} }
\newcommand{\ep}{\varepsilon}
\newcommand{\fy}{\varphi}
\newcommand{\cdo}{\, \cdot \, }

\newcommand{\wpr}{{\text{\footnotesize $\#$}}}

\newcommand{\rank}{\operatorname{rank}}
\newcommand{\eabs}[1]{\langle #1\rangle}     

\newcommand{\vrum}{\vspace{0.2cm}}

\newcommand{\mascB}{\mathscr B}
\newcommand{\mascC}{\mathscr C}
\newcommand{\mascH}{\mathscr H}
\newcommand{\mascI}{\mathscr I}
\newcommand{\mascL}{\mathscr L}

\newcommand{\mascS}{\mathscr S}
\newcommand{\maclB}{\mathcal B}
\newcommand{\maclK}{\mathcal K}
\newcommand{\maclM}{\mathcal M}
\newcommand{\maclS}{\mathcal S}

\numberwithin{equation}{section}          

\newtheorem{thm}{Theorem}
\numberwithin{thm}{section}
\newtheorem{prop}[thm]{Proposition}

\newtheorem{lemma}[thm]{Lemma}

\newcommand{\rubrik}{}

\theoremstyle{definition}

\newtheorem{defn}[thm]{Definition}

\theoremstyle{remark}
\newtheorem{rem}[thm]{Remark}

\title{Decompositions of Gelfand-Shilov kernels into kernels of similar class}

\author{Joachim Toft}

\address{Department of Computer science, Physics and Mathematics,
Linn{\ae}us University, V{\"a}xj{\"o}, Sweden}

\email{joachim.toft@lnu.se}

\author{Andrei Khrennikov}

\address{Department of Computer science, Physics and Mathematics,
Linn{\ae}us University, V{\"a}xj{\"o}, Sweden}

\email{andrei.khrennikov@lnu.se}

\author{B{\"o}rje Nilsson}

\address{Department of Computer science, Physics and Mathematics,
Linn{\ae}us University, V{\"a}xj{\"o}, Sweden}

\email{borje.nilsson@lnu.se}

\author{Sven Nordebo}

\address{Department of Computer science, Physics and Mathematics,
Linn{\ae}us University, V{\"a}xj{\"o}, Sweden}

\email{sven.nordebo@lnu.se}

\keywords{matrices, Hermite functions, kernel theorems, Schatten-von
Neumann operators, singular values}

\subjclass{}

\begin{document}

\begin{abstract}
We prove that any linear operator with kernel in a Gelfand-Shilov space is a composition
of two operators with kernels in the same Gelfand-Shilov space. We also give links on
numerical approximations for such compositions. We apply these composition rules
to establish Schatten-von Neumann properties for such operators.
\end{abstract}

\maketitle

\par

\section{Introduction}\label{sec0}

\par

In this paper we investigate possibilities to decompose linear operators
into operators in the same class. It is obvious that for any topological vector space
$\maclB$, the set $\maclM = \mascL (\maclB )$ of linear and continuous
operators on $\maclB$ is a \emph{decomposition algebra}. That is, any operator
$T$ in $\maclM$ is a composition of two operators in $T_1,T_2\in \maclM$,
since we may choose $T_1$ as the identity operator, and $T_2=T$.
If in addition $\maclB$ is a Hilbert space, then it follows from spectral decomposition
that the set of compact operators on $\maclB$ is a decomposition algebra,
where the decomposition property are obtained by
straight-forward applications of the spectral theorem.

\par

An interesting subclass of linear and continuous operators on $L^2$ concerns
the set of all linear operators whose kernels belong to the Schwartz space. There
are several proofs of the fact that this operator class is a decomposition algebra
(cf. e.g. \cite{Beals, Kl, Sad, Vo} and the references therein).

\par

We remark that there are operator algebras which are not decomposition
algebras. For example, if $\maclB$  is an infinite-dimensional Hilbert space
and $0<p<\infty$, then the set of all Schatten-von Neumann operators of order
$p$ is not a decomposition algebra.

\par

In this paper we consider the case when $\maclM$ is the set of all linear operators
with distribution kernels in the Schwartz class, or in a Gelfand-Shilov
space. We note that these operator classes are small, because the restrictions
on corresponding kernels are strong. For example, it is obvious that
the identity operator does not belong to any of these operator classes.

\par

We prove that any such $\maclM$ is a decomposition algebra. Furthermore,
in the end of the paper we apply the result and prove that any operator
in $\maclM$ as a map between appropriate Banach spaces, or quasi-Banach spaces,
belongs to every Schatten-von Neumann class between these spaces.

\par

\section{Preliminaries}\label{sec1}

\par

In this section we recall some facts on Gelfand-Shilov spaces and
pseudo-differential operators.

\par

We start by defining Gelfand-Shilov spaces and recalling some basic facts.

\par

Let $0<h,s\in \mathbf R$ be fixed. Then we let $\mathcal S_{s,h}(\rr d)$
be the set of all $f\in C^\infty (\rr d)$ such that
\begin{equation*}
\nm f{\mathcal S_{s,h}}\equiv \sup \frac {|x^\beta \partial ^\alpha
f(x)|}{h^{|\alpha | + |\beta |}(\alpha !\, \beta !)^s}
\end{equation*}
is finite. Here the supremum should be taken over all $\alpha ,\beta \in
\mathbf N^d$ and $x\in \rr d$.

\par

Obviously $\mathcal S_{s,h}\subseteq
\mathscr S$ is a Banach space which increases with $h$ and $s$.
Furthermore, if $s\ge 1/2$ and $h$ is sufficiently large, then $\mathcal
S_{s,h}$ contains all finite linear combinations of Hermite functions.
Since such linear combinations are dense in $\mathscr S$, it follows
that the dual $\mathcal S_{s,h}'(\rr d)$ of $\mathcal S_{s,h}(\rr d)$ is
a Banach space which contains $\mathscr S'(\rr d)$.

\par

The \emph{Gelfand-Shilov spaces} $\mathcal S_s(\rr d)$ and
$\Sigma _s(\rr d)$ are the inductive and projective limits respectively
of $\mathcal S_{s,h}(\rr d)$. This means that
\begin{equation}\label{GSspacecond1}
\mathcal S_s(\rr d) = \bigcup _{h>0}\mathcal S_{s,h}(\rr d)
\quad \text{and}\quad \Sigma _s(\rr d) =\bigcap _{h>0}\mathcal S_{s,h}(\rr d),
\end{equation}
$\mathcal S_s(\rr d)$ is equipped with the strongest topology such
that each inclusion map from $\mathcal S_{s,h}(\rr d)$ to $\mathcal S_s(\rr d)$
is continuous, and that $\Sigma _s(\rr d)$ is a Fr{\'e}chet space with semi norms
$\nm \cdo{\mathcal S_{s,h}}$, $h>0$.

\par

We remark that $\mathcal S_s(\rr d)$ equals $\{ 0\}$ if and only
if $s<1/2$, and that $\Sigma _s(\rr d)$ equals $\{ 0 \}$ if and only
if $s\le 1/2$. For each $\ep >0$ and $s\ge 1/2$, we have
$$
\Sigma _s (\rr d)\hookrightarrow \mathcal S_s(\rr d)\hookrightarrow
\Sigma _{s+\ep}(\rr d).
$$
Here, if $A$ and $B$ are topological spaces, then $A\hookrightarrow B$
means that $A$ is continuously embedded in $B$.
On the other hand, in \cite{pil} there is an alternative elegant definition of
$\Sigma _{s_1}(\rr d)$ and $\mathcal S _{s_2}(\rr d)$ such that these spaces
agrees with the definitions above when $s_1>1/2$ and $s_2\ge 1/2$, but
$\Sigma _{1/2}(\rr d)$ is non-trivial and contained in $\mathcal S_{1/2}(\rr d)$.

\par

From now on we assume that $s>1/2$ when considering $\Sigma _s(\rr d)$.

\medspace

The \emph{Gelfand-Shilov distribution spaces} $\mathcal S_s'(\rr d)$
and $\Sigma _s'(\rr d)$ are the projective and inductive limit
respectively of $\mathcal S_{s,h}'(\rr d)$.  This means that
\begin{equation}\tag*{(\ref{GSspacecond1})$'$}
\mathcal S_s'(\rr d) = \bigcap _{h>0}\mathcal S_{s,h}'(\rr d)\quad
\text{and}\quad \Sigma _s'(\rr d) =\bigcup _{h>0} \mathcal S_{s,h}'(\rr d).
\end{equation}
We remark that already in \cite{GS} it is proved that $\mathcal S_s'(\rr d)$
is the dual of $\mathcal S_s(\rr d)$, and if $s>1/2$, then $\Sigma _s'(\rr d)$
is the dual of $\Sigma _s(\rr d)$ (also in topological sense).

\par

The Gelfand-Shilov spaces are invariant under several basic transformations.
For example they are invariant under translations, dilations, tensor products
and to some extent under any Fourier transformation.

\par

From now on we let $\mathscr F$ be the Fourier transform which takes the form
$$
(\mathscr Ff)(\xi )= \widehat f(\xi ) \equiv (2\pi )^{-d/2}\int _{\rr
{d}} f(x)e^{-i\scal  x\xi }\, dx
$$
when $f\in L^1(\rr d)$. Here $\scal \cdo \cdo$ denotes the usual scalar product
on $\rr d$. The map $\mathscr F$ extends 
uniquely to homeomorphisms on $\mathscr S'(\rr d)$, $\mathcal S_s'(\rr d)$
and $\Sigma _s'(\rr d)$, and restricts to 
homeomorphisms on $\mathscr S(\rr d)$, $\mathcal S_s(\rr d)$ and $\Sigma _s(\rr d)$, 
and to a unitary operator on $L^2(\rr d)$.

\par

The following lemma shows that elements in Gelfand-Shilov spaces can
be characterized by estimates of the form
\begin{equation}\label{GSexpcond}
|f(x)|\le Ce^{-\ep |x|^{1/s}}\quad \text{and}\quad |\widehat f (\xi )|\le Ce^{-\ep |\xi |^{1/s}} .
\end{equation}
The proof is omitted, since the result can be found in e.{\,}g. \cite{GS, ChuChuKim}.

\par

\begin{lemma}\label{GSFourierest}
Let $f\in \mathcal S'_{1/2}(\rr d)$. Then the following is true:
\begin{enumerate}
\item if $s\ge 1/2$, then $f\in \mathcal S_s(\rr d)$, if and only if there
are constants $\ep >0$ and $C>0$
such that \eqref{GSexpcond} holds;

\vrum

\item  if $s>1/2$, then $f\in \Sigma _s(\rr d)$, if and only if for each $\ep >0$,
there is a constant $C$ such that \eqref{GSexpcond} holds.
\end{enumerate}
\end{lemma}

\par

Gelfand-Shilov spaces can also easily be characterized by Hermite functions. We recall
that the Hermite function $h_\alpha$ with respect to the multi-index $\alpha \in \mathbf N^d$ is defined by
$$
h_\alpha (x) = \pi ^{-d/4}(-1)^{|\alpha |}(2^{|\alpha |}\alpha
!)^{-1/2}e^{|x|^2/2}(\partial ^\alpha e^{-|x|^2}).
$$
The set $(h_\alpha )_{\alpha \in \mathbf N^d}$ is an orthonormal basis for $L^2(\rr d)$. In particular,
\begin{equation}\label{hermexp}
f=\sum _\alpha c_\alpha h_\alpha ,\quad c_\alpha =(f,h_\alpha )_{L^2(\rr d)},
\end{equation}
and
$$
\nm f{L^2}=\nm {\{ c_\alpha \} _\alpha }{l^2}<\infty ,
$$
when $f\in L^2(\rr d)$. Here and in what follows, $(\cdo ,\cdo )_{L^2(\rr d)}$ denotes any
continuous extension of the $L^2$ form on $\mathcal S_{1/2}(\rr d)$.

\par

Let $p\in [1,\infty ]$ be fixed. Then it is well-known that $f$ here belongs to
$\mathscr S(\rr d)$, if and only if
\begin{equation}\label{Shermexp}
\nm {\{ c_\alpha \eabs \alpha ^t\} _\alpha }{l^p}<\infty 
\end{equation}
for every $t\ge 0$. Here we let $\eabs x=(1+|x|^2)^{1/2}$
when $x\in \rr d$. Furthermore, for every $f\in \mathscr S'(\rr d)$, the expansion
\eqref{hermexp}
still holds, where the sum converges in $\mathscr S'$, and \eqref{Shermexp} holds for some
choice of $t\in \mathbf R$, which depends on $f$.

\par

The following proposition, which can be found in e.{\,}g. \cite{GrPiRo}, shows that similar
conclusion for Gelfand-Shilov spaces holds, after the estimate \eqref{Shermexp} is replaced by
\begin{equation}\label{GShermexp}
\nm { \{ c_\alpha e^{t |\alpha |^{1/2s}} \} _\alpha }{l^p}<\infty .
\end{equation}
We refer to the proof of formula (2.12) in \cite{GrPiRo} for its proof.

\par

\begin{prop}\label{stftGS2}
Let $p\in [1,\infty ]$, $f\in \mathcal S_{1/2}'\rr d)$, $s_1\ge 1/2$, $s_2>1/2$ and let
$c_\alpha$ be as in \eqref{hermexp}. Then the following is true:
\begin{enumerate}
\item $f\in \mathscr S(\rr d)$, if and only if \eqref{Shermexp} holds for every $t>0$.
Furthermore, \eqref{hermexp} holds where the sum converges in $\mathscr S$;

\vrum

\item $f\in \mathcal S_{s_1}(\rr d)$, if and only if \eqref{GShermexp} holds for some $t>0$.
Furthermore, \eqref{hermexp} holds where the sum converges in $\mathcal S_{s_1}$;

\vrum

\item  $f\in \Sigma _{s_2}(\rr d)$, if and only if \eqref{GShermexp} holds for every $t>0$.
Furthermore, \eqref{hermexp} holds where the sum converges in $\Sigma _{s_2}$.

\vrum

\item $f\in \mathscr S'(\rr d)$, if and only if \eqref{Shermexp} holds for some $t<0$.
Furthermore, \eqref{hermexp} holds where the sum converges in $\mathscr S'$;

\vrum

\item $f\in \mathcal S_{s_1}'(\rr d)$, if and only if \eqref{GShermexp} holds for every
$t<0$. Furthermore, \eqref{hermexp} holds where the sum converges in $\mathcal S_{s_1}'$;

\vrum

\item  $f\in \Sigma _{s_2}'(\rr d)$, if and only if \eqref{GShermexp} holds for some $t<0$.
Furthermore, \eqref{hermexp} holds where the sum converges in $\Sigma _{s_2}'$;
\end{enumerate}
\end{prop}

\par

Proposition \ref{stftGS2} is fundamental in the proof of Theorem \ref{GSkernels} which in turn
is a cornerstone in the proof of Theorem \ref{thmSchattGSkernels}. It also give
links on how properties valid for tempered distributions or Schwartz functions
can be carried over to Gelfand-Shilov spaces of functions or distributions by
passing from estimates of the form \eqref{Shermexp}
to estimates of the form \eqref{GShermexp}, and vice versa.

\medspace

Next we recall some properties of pseudo-differential operators. Let
$t\in \mathbf R$ be fixed and let $a\in \mathcal S_{1/2}(\rr {2d})$. Then the
pseudo-differential operator $\op _t(a)$
with symbol $a$ is the linear and continuous operator on $\mathcal S_{1/2}(\rr d)$,
defined by the formula
\begin{equation}\label{e0.5}
(\op _t(a)f)(x)
=
(2\pi  ) ^{-d}\iint a((1-t)x+ty,\xi )f(y)e^{i\scal {x-y}\xi }\,
dyd\xi .
\end{equation}
The definition of $\op _t(a)$ extends to each $a\in \mathcal S_{1/2}'(\rr
{2d})$, and then $\op _t(a)$ is continuous from $\mathcal S_{1/2}(\rr d)$ to
$\mathcal S_{1/2}'(\rr d)$. (Cf. e.{\,}g. \cite {CPRT10}, and to some extent \cite{Ho1}.) 
More precisely, for
any $a\in \mathcal S_{1/2}'(\rr {2d})$, the operator $\op _t(a)$
is defined as the linear and continuous
operator from $\mathcal S_{1/2}(\rr d)$ to $\mathcal S_{1/2}'(\rr d)$ with
distribution kernel given by
\begin{equation}\label{atkernel}
K_{a,t}(x,y)=(\mathscr F_2^{-1}a)((1-t)x+ty,x-y).
\end{equation}
Here $\mathscr F_2F$ is the partial Fourier transform of $F(x,y)\in
\mathcal S_{1/2}'(\rr {2d})$ with respect to the $y$ variable. This
definition makes
sense, since the mappings $\mathscr F_2$ and $F(x,y)\mapsto
F((1-t)x+ty,y-x)$ are homeomorphisms on $\mathcal S_{1/2}'(\rr
{2d})$.

\par

On the other hand, let $T$ be an arbitrary linear and continuous
operator from $\maclS _{1/2}(\rr d)$ to $\maclS _{1/2}'(\rr d)$. Then it follows from
Theorem 2.2 in \cite{LozPerTask} that for some
$K =K_T\in \maclS '_{1/2}(\rr {2d})$ we have
$$
(Tf,g)_{L^2(\rr d)} = (K,g\otimes \overline f )_{L^2(\rr {2d})},
$$
for every $f,g\in \maclS _{1/2}(\rr d)$. Now by letting $a$ be defined by
\eqref{atkernel} after replacing $K_{a,t}$ with $K$ it follows that
$T=\op _t(a)$. Consequently, the map $a\mapsto \op _t(a)$ is bijective from
$\maclS _{1/2}'(\rr {2d})$ to $\mathscr L(\maclS _{1/2}(\rr d),\maclS _{1/2}'(\rr d))$.

\par

If $t=1/2$, then $\op _t(a)$ is equal to the Weyl
quantization $\op ^w(a)$ of $a$. If instead $t=0$, then the standard
(Kohn-Nirenberg) representation $a(x,D)$ is obtained.

\par

In particular, if $a\in \maclS _{1/2}'(\rr {2d})$ and $s,t\in
\mathbf R$, then there is a unique $b\in \maclS _{1/2}'(\rr {2d})$ such that
$\op _s(a)=\op _t(b)$. By straight-forward applications of Fourier's
inversion  formula, it follows that
\begin{equation}\label{pseudorelation}
\op _s (a)=\op _t(b) \quad \Longleftrightarrow \quad b(x,\xi )=e^{i(t-s)\scal
{D_x}{D_\xi}}a(x,\xi ).
\end{equation}
(Cf. Section 18.5 in \cite{Ho1}.) Note here that the right-hand side makes sense,
because $e^{i(t-s)\scal
{D_x }{D_\xi }}$ on the Fourier transform side is a multiplication by
the function $e^{i(t-s)\scal x \xi }$, which is a continuous
operation on $\maclS _{1/2}'(\rr {2d})$, in view of the definitions.

\par

Next let $t\in \mathbf R$ be fixed and let $a,b\in \maclS _{1/2}'(\rr {2d})$. Then the product
$a\wpr _tb$ is defined by the formula
$$
\op _t(a\wpr _tb) = \op _t(a)\circ \op _t(b), 
$$
provided the right-hand side makes sense as a continuous operator from
$\maclS _{1/2}(\rr d)$ to $\maclS _{1/2}'(\rr d)$. We note that the element $a\wpr _t b$ is uniquely defined
and belongs to $c\in \maclS _{1/2}'(\rr {2d})$.

\par

\section{Gelfand-Shilov kernels and pseudo-differential
operators}\label{sec2}

\par

In what follows we use the convention that if $T_0$ is a linear and continuous
operator from $\maclS _{1/2}(\rr {d_1})$ to $\maclS _{1/2}'(\rr {d_2})$,
and $g \in \maclS _{1/2}'(\rr {d_0})$, then $T_0\otimes g$ is the linear
and continuous operator from $\maclS _{1/2}(\rr {d_1})$ to
$\maclS _{1/2}'(\rr {d_2+d_0})$, given by
$$
(T_0\otimes g)\, :\, f \mapsto (T_0f)\otimes g .
$$

\par

In the following definition we recall that an operator $T$ from $\maclS _{1/2}(\rr d)$ to
$\maclS _{1/2}'(\rr d)$ is called \emph{positive semi-definite}, if
$(Tf,f)_{L^2}\ge 0$, for every $f\in \maclS _{1/2}(\rr d)$. Then we write $T\ge 0$.

\par

\begin{defn}\label{defHermdiagform}
Let $d_2\ge d_1$ and let $T$ be a linear operator from $\maclS _{1/2}(\rr {d_1})$ to
$\maclS _{1/2}'(\rr {d_2})$. Then $T$ is said to be a \emph{Hermite diagonal operator} if $T=T_0\otimes g$,
where the Hermite functions are eigenfunctions to $T_0$, and either
$d_2=d_1$ and $g=1$, or $d_2>d_1$ and $g$ is a Hermite function.

\par

Moreover, if $T=T_0\otimes g$ is on Hermite diagonal form and $T_0$ is
positive semi-definite, then
$T$ is said to be a \emph{positive Hermite diagonal operator}.
\end{defn}

\par

The first part of the following result can be found in \cite{Beals, Vo}
(see also \cite{Kl, Sad} and the references therein for an elementary proof).

\par

\begin{thm}\label{GSkernels}
Let $T$ be a linear and continuous operator from $\maclS _{1/2}(\rr {d_1})$
to $\maclS _{1/2}'(\rr {d_2})$ with the kernel $K$, and let $d_0\ge \min (d_1,d_2)$.
Then the following is true:
\begin{enumerate}
\item If $K\in \mathscr S (\rr {d_2+d_1})$, then there
are operators $T_1$ and $T_2$ with
kernels $K_1\in \mathscr S (\rr {d_0+d_1})$ and $K_2\in \mathscr S (\rr {d_2+d_0})$
respectively such that $T=T_2\circ T_1$. Furthermore, at least one
of $T_1$ and $T_2$ can be chosen as positive Hermite diagonal operator;

\vrum

\item If $s\ge 1/2$ and $K\in \maclS _s(\rr {d_2+d_1})$, then there
are operators $T_1$ and $T_2$ with
kernels $K_1\in \maclS _s(\rr {d_0+d_1})$ and $K_2\in \maclS _s(\rr {d_2+d_0})$
respectively such that $T=T_2\circ T_1$. Furthermore, at least one
of $T_1$ and $T_2$ can be chosen as positive Hermite diagonal operator;

\vrum

\item If $s>1/2$ and $K\in \Sigma _s(\rr {d_2+d_1})$, then there
are operators $T_1$ and $T_2$ with
kernels $K_1\in \Sigma _s(\rr {d_0+d_1})$ and $K_2\in \Sigma
_s(\rr {d_2+d_0})$ respectively such that $T=T_2\circ T_1$.  Furthermore,
at least one of $T_1$ and $T_2$ can be chosen as positive Hermite diagonal operator.
\end{enumerate}
\end{thm}

\par

\begin{rem}
An operator with kernel in $\maclS _s(\rr {2d})$ is sometimes called a regularizing
operator with respect to $\maclS _s$, because it extends uniquely to a continuous
map from (the large space) $\maclS _s'(\rr d)$ into (the small space) $\maclS
_s(\rr d)$. Analogously, an operator with kernel in $\Sigma _s(\rr {2d})$
($\mathscr S (\rr {2d})$) is sometimes called a regularizing operator with
respect to $\Sigma _s$ ($\mathscr S$).
\end{rem}

\par

\begin{proof}[Proof of Theorem \ref{GSkernels}]
We only prove (2) and (3). The assertion (1) follows by similar arguments as in the proof of (3).
Furthermore, a proof of the first part of (1) can be found in \cite{Beals,Sad}.

\par

First we assume that $d_0= d_1$, and start to prove (2). Let $h_{d,\alpha}(x)$
be the Hermite function on $\rr d$ of order
$\alpha \in \mathbf N^d$. Then $K$ posses the expansion
\begin{equation}\label{Kexp}
K(x,y) = \sum _{\alpha \in \mathbf N^{d_2}}\sum _{\beta \in \mathbf N^{d_1}}
a_{\alpha ,\beta}h_{d_2,\alpha}(x)h_{d_1,\beta}(y),
\end{equation}
where the coefficients $a_{\alpha ,\beta}$ satisfies
\begin{equation}\label{coeffsest}
\sup _{\alpha ,\beta} |a_{\alpha ,\beta}e^{r(|\alpha |^{1/2s}+|\beta |^{1/2s})}| <\infty ,
\end{equation}
for some $r>0$.

\par

Now we let $z\in \rr {d_1}$, and
\begin{equation}\label{K1K2def}
\begin{aligned}
K_{0,2}(x,z) &= \sum _{\alpha \in \mathbf N^{d_2}}\sum _{\beta \in \mathbf N^{d_1}}
b_{\alpha ,\beta}h_{d_2,\alpha}(x)h_{d_1,\beta}(z),
\\[1ex]
K_{0,1}(z,y) &= \sum _{\alpha ,\beta \in \mathbf N^{d_1}}
c_{\alpha ,\beta }h_{d_1,\alpha}(z)h_{d_1,\beta}(y),
\end{aligned}
\end{equation}
where
$$
b_{\alpha ,\beta} = a_{\alpha ,\beta} e^{r|\beta|^{1/2s}/2}\quad \text{and}\quad
c_{\alpha ,\beta} = \delta _{\alpha ,\beta}e^{-r|\alpha |^{1/2s}/2}.
$$
Here $\delta _{\alpha ,\beta}$ is the Kronecker delta. Then it follows that
$$
\int K_{0,2}(x,z)K_{0,1}(z,y)\, dz = \sum _{\alpha \in \mathbf N^{d_2}}\sum _{\beta
\in \mathbf N^{d_1}} a_{\alpha ,\beta}h_{d_2,\alpha}(x)h_{d_1,\beta}(y)
= K(x,y).
$$
Hence, if $T_j$ is the operator with kernel $K_{0,j}$, $j=1,2$, then $T=T_2\circ T_1$.
Furthermore,
$$
\sup _{\alpha ,\beta} |b_{\alpha ,\beta}e^{r(|\alpha |^{1/2s}+|\beta |^{1/2s})/2}|
\le \sup _{\alpha ,\beta} |a_{\alpha ,\beta}e^{r(|\alpha |^{1/2s}+|\beta |^{1/2s})}|
<\infty 
$$
and
$$
\sup _{\alpha ,\beta} |c_{\alpha ,\beta}e^{r(|\alpha |^{1/2s}+|\beta |^{1/2s}/2}|
= \sup _{\alpha } |e^{-r|\alpha |^{1/2s}/2}e^{r|\alpha |^{1/2s}/2}|
<\infty .
$$
This implies that $K_{0,1}\in \maclS _s(\rr {d_1+d_1})$ and
$K_{0,2}\in \maclS _s(\rr {d_2+d_1})$ in view of Proposition \ref{stftGS2},
and (2) follows with $K_1=K_{0,1}$ and $K_2=K_{0,2}$, in the case
$d_0=d_1$.

\par

In order to prove (3), we assume that $K\in \Sigma _s(\rr {d_2+d_1})$, and we
let $a_{\alpha ,\beta}$ be the same as the above. Then
\eqref{coeffsest} holds for any $r>0$, which implies that if $N\ge 0$ is an
integer, then
\begin{equation}\label{ThetaNset}
\Theta _N \equiv \sup \sets {|\beta |}{|a_{\alpha ,\beta}|\ge
e^{-2(N+1)(|\alpha |^{1/2s}+|\beta |^{1/2s})}\ \text{for some}\ \alpha
\in \mathbf N^{d_2}}
\end{equation}
is finite.

\par

We let
\begin{align*}
I_1 &= \sets {\beta \in \mathbf N^{d_1}}{|\beta |\le \Theta _1+1}
\intertext{and define inductively}
I_j &= \sets {\beta \in \mathbf N^{d_1}\setminus I_{j-1}}{|\beta |\le \Theta _j+j},
\quad j\ge 2.
\end{align*}
Then
$$
I_j\cap I_k=\emptyset \quad \text{when}\quad j\neq k,\quad \text{and}\quad
\bigcup _{j\ge 0} I_j=\mathbf N^{d_1}.
$$

\par

We also let $K_{0,1}$ and $K_{0,2}$ be given by \eqref{K1K2def}, where
$$
b_{\alpha _2,\beta} = a_{\alpha _2,\beta}e^{j|\beta |^{1/2s}}\quad \text{and}\quad
c_{\alpha _1,\beta} = \delta _{\alpha _1,\beta}e^{-j|\beta |^{1/2s}},
$$
when $\alpha _1\in \mathbf N^{d_1}$, $\alpha _2\in \mathbf N^{d_2}$
and $\beta \in I_j$. If $T_j$ is the operator with kernel $K_{0,j}$ for $j=1,2$, then it follows
that $T_2\circ T_1 =T$. Furthermore, if $r>0$, then we have
$$
\sup _{\alpha ,\beta } |b_{\alpha ,\beta} e^{r (|\alpha |^{1/2s}+|\beta |^{1/2s})}|
\le J_1+J_2,
$$
where
\begin{align}
J_1 &= \sup _{j\le r+1}\sup _\alpha \sup _{\beta \in I_j}|b_{\alpha ,\beta}
e^{r (|\alpha |^{1/2s}+|\beta |^{1/2s})}|\label{J1def}
\intertext{and}
J_2 &= \sup _{j> r+1}\sup _\alpha \sup _{\beta \in I_j}|b_{\alpha ,\beta}
e^{r (|\alpha |^{1/2s}+|\beta |^{1/2s})}|\label{J2def}
\end{align}

\par

Since only finite numbers of $\beta$ is involved in the suprema in \eqref{J1def},
it follows from \eqref{coeffsest} and the definition of $b_{\alpha ,\beta}$ that $J_1$
is finite.

\par

For $J_2$ we have
\begin{multline*}
J_2 = \sup _{j> r+1}\sup _\alpha \sup _{\beta \in I_j}|a_{\alpha ,\beta}
e^{r|\alpha |^{1/2s}+(r+j)|\beta |^{1/2s})}|
\\[1ex]
\le \sup _{j> r+1}\sup _\alpha \sup _{\beta \in I_j}|e^{-2j(|\alpha |^{1/2s}+|\beta|^{1/2s})}
e^{r|\alpha |^{1/2s}+(r+j)|\beta |^{1/2s})}| <\infty ,
\end{multline*}
where the first inequality follows from \eqref{ThetaNset}. Hence
$$
\sup _{\alpha ,\beta } |b_{\alpha ,\beta} e^{r (|\alpha |^{1/2s}+|\beta |^{1/2s})}|< \infty ,
$$
which implies that $K_{0,2}\in \Sigma _s(\rr {d_2+d_1})$.

\par

If we now replace $b_{\alpha ,\beta}$ with $c_{\alpha ,\beta}$ in the definition
of $J_1$ and $J_2$, it follows by similar arguments that both $J_1$ and $J_2$
are finite, also in this case. This gives
$$
\sup _{\alpha ,\beta } |c_{\alpha ,\beta} e^{r (|\alpha |^{1/2s}+|\beta |^{1/2s})}|< \infty .
$$
Hence $K_1\in \Sigma _s(\rr {d_1+d_1})$, and (3) follows in the case $d_0=d_1$.

\par

Next assume that $d_0>d_1$, and let $d=d_0-d_1\ge 1$. Then we set
$$
K_1(z,y) = K_{0,1}(z_1,y)h_{d,0}(z_2)\quad \text{and}\quad
K_2(x,z) = K_{0,2}(x,z_1)h_{d,0}(z_2),
$$
where $K_{0,j}$ are the same as in the first part of the proofs, $z_1\in \rr {d_1}$
and $z_2\in \rr d$, giving that $z=(z_1,z_2)\in \rr {d_0}$. We get
$$
\int _{\rr {d_0}}K_2(x,z)K_1(z,y)\, dz = \int _{\rr {d_1}}K_{0,2}(x,z_1)K_{0,1}(z_1,y)\,
dz_1 = K(x,y).
$$
The assertions (2) now follows in the case $d_0>d_1$ from the equivalences
\begin{alignat*}{3}
K_1 &\in \maclS _s(\rr {d_0+d_1})&\quad &\Longleftrightarrow & \quad
K_{0,1} &\in \maclS _s(\rr {d_1+d_1})
\intertext{and}
K_2 &\in \maclS _s(\rr {d_2+d_0})&\quad &\Longleftrightarrow & \quad
K_{0,2} &\in \maclS _s(\rr {d_2+d_1}),
\end{alignat*}
Since the same equivalences hold after the $\maclS _s$ spaces have been replaced
by $\Sigma _s$ spaces, the assertion (3) also follows in the case $d_0>d_1$, and the
theorem follows in the case $d_0\ge d_1$.

\par

It remains to prove the result in the case $d_0\ge d_2$. The rules of $d_1$ and
$d_2$ are interchanged when taking the adjoints. Hence, the result follows
from the first part of the proof in combination with the facts that
$\maclS _s$ and $\Sigma _s$ are invariant under pullbacks of bijective linear
transformations. The proof is complete.
\end{proof}

\par

\begin{rem}
From the construction of $K_1$ and $K_2$ in the proof of Theorem \ref{GSkernels}, it follows
that it is not so complicated for using numerical methods when
obtaining approximations of candidates to $K_1$ and $K_2$. In fact, $K_1$ and $K_2$ are formed
explicitly by the elements of the matrix for $T$, when the Hermite functions are used as
basis for $\mathscr S$, $\maclS _s$ and $\Sigma _s$.
\end{rem}

\par

The following result is an immediate consequence of Theorem \ref{GSkernels} and
the fact that the map $a\mapsto K_{a,t}$ is continuous and bijective on $\maclS
_{s_1}(\rr {2d})$, and on $\Sigma _{s_2}(\rr {2d})$, for every $s_1\ge 1/2$, $s_2>1/2$
and $t\in \mathbf R$.

\par

\begin{thm}\label{GSpseudo}
Let $t\in \mathbf R$, $s_1\ge 1/2$ and $s_2>1/2$. Then the following is true:
\begin{enumerate}
\item if $a\in \maclS _{s_1}(\rr {2d})$, then there are $a_1,a_2\in \maclS _{s_1}(\rr {2d})$
such that $a=a_1\wpr _ta_2$;

\vrum

\item  if $a\in \Sigma _{s_2}(\rr {2d})$, then there are $a_1,a_2\in \maclS _{s_2}(\rr {2d})$
such that $a=a_1\wpr _ta_2$.
\end{enumerate}
\end{thm}

\par

%
%

\par

\section{Schatten-von Neumann properties for operators with
Gelfand-Shilov kernels}\label{sec3}

\par

In this section we use Theorem \ref{GSkernels} to prove that if $T$ is a
linear operator with kernel in
$\maclS _s$, and $\mascB _1,\mascB _2\subseteq \maclS _s'$ are such that
$\maclS _s\subseteq \mascB _1,\mascB _2$, then $T$ belongs to any
Schatten-von Neumann class of operators between $\mascB _1$
and $\mascB _2$. In particular it follows that
the singular values of $T$ fulfill strong decay properties.

\par

We start by recalling the definition of Schatten-von Neumann operators
in the (quasi-)Banach space case. Let $\mascB$ be a vector space. A
\emph{quasi-norm} $\nm \cdo {\mascB}$ on $\mascB$ is a non-negative and
real-valued function on $\mascB$ which is non-degenerate in the sense
$$
\nm {f}{\mascB}=0\qquad \Longleftrightarrow \qquad f=0,\quad f\in \mascB,
$$
and fulfills
\begin{equation}\notag
\begin{alignedat}{2}
\nm {\alpha f}{\mascB} &= |\alpha |\cdot \nm f{\mascB}, &\qquad 
f &\in \mascB ,\ \alpha \in \mathbf C
\\[1ex]
\text{and}\qquad \nm {f+g}{\mascB} &\le D(\nm f{\mascB} +\nm g{\mascB}),
&\qquad f,g &\in \mascB , 
\end{alignedat}
\end{equation}
for some constant $D\ge 1$ which is independent of $f,g\in \mascB$. Then
$\mascB$ is a topological vector space when the topology for $\mascB$
is defined by $\nm \cdo{\mascB}$, and $\mascB$ is called a quasi-Banach
space if $\mascB$ is complete under this topology.

\par

Let $\mascB _1$ and $\mascB _2$ be (quasi-)Banach spaces, and let
$T$ be a linear map between $\mascB _1$ and $\mascB _2$. For
every integer $j\ge 1$, the \emph{singular values} of order $j$ of
$T$ is given by
$$
\sigma _j(T) = \sigma _j(\mascB _1,\mascB _2,T)
\equiv \inf \nm {T-T_0}{\mascB _1\to \mascB _2},
$$
where the infimum is taken over all linear operators $T_0$ from $\mascB _1$
to $\mascB _2$ with rank at most $j-1$. Therefore, $\sigma _1(T)$
equals $\nm T{\mascB _1\to \mascB _2}$, and $\sigma _j(T)$ are 
non-negative which decreases with $j$.

\par

For any $p\in (0,\infty ]$ we set
$$
\nm T{\mascI _p} = \nm T{\mascI _p(\mascB _1,\mascB _2)}
\equiv \nm { ( \sigma _j(\mascB _1,\mascB _2,T) ) _{j=1}^\infty}{l^p}
$$
(which might attain $+\infty$). The operator $T$ is called a \emph{Schatten-von
Neumann operator} of order $p$ from $\mascB _1$ to $\mascB _2$, if
$\nm T{\mascI _p}$ is finite, i.{\,}e.
$( \sigma _j(\mascB _1,\mascB _2,T) ) _{j=1}^\infty$ should belong to $l^p$.
The set of all Schatten-von Neumann operators of order $p$ from
$\mascB _1$ to $\mascB _2$ is denoted by $\mascI _p =
\mascI _p(\mascB _1,\mascB _2)$. We note that
$\mascI _\infty(\mascB _1,\mascB _2)$ agrees with $\maclB (\mascB _1
,\mascB _2)$, the set of linear and bounded operators
from $\mascB _1$ to $\mascB _2$, and if $p<\infty$, then
$\mascI _p(\mascB _1,\mascB _2)$ is contained in $\maclK(\mascB _1
,\mascB _2)$, the set of linear and compact operators from $\mascB _1$
to $\mascB _2$. If $\mascB _1=\mascB _2$,
then the shorter notation $\mascI  _p(\mascB _1)$ is used
instead of $\mascI  _p(\mascB _1,\mascB _2)$, and similarly
for $\maclB(\mascB _1,\mascB _2)$ and $\maclK(\mascB _1,
\mascB _2)$.

\par

Schatten-von Neumann classes posses several convenient properties.
For example, if $\mascB _1$, $\mascB _2$ and $\mascB _3$ are Banach
spaces, $p_1,p_2,r\in (0,\infty ]$ satisfy the
H{\"o}lder condition $1/p_1+1/p_2=1/r$,
and $T_k\in \mascI _{p_k}(\mascB _k,\mascB _{k+1})$, then
$T_2\circ T_1\in \mascI _{r}(\mascB _1,\mascB _3)$, and
$$
\nm {T_2\circ T_1}{\mascI _{r}(\mascB _1,\mascB _3)}\le C_r
\nm {T_1}{\mascI _{p_1}(\mascB _1,\mascB _2)}
\nm {T_2}{\mascI _{p_2}(\mascB _2,\mascB _3)},
$$
where $C_r=1$ when $\mascB _j$, $j=1,2,3$, are Hilbert spaces, and $C_r=2^{1/r}$
otherwise (cf. e.{\,}g. \cite{To12, Pie}). We refer to \cite{Si, Pie} for a brief analysis of
Schatten-von Neumann operators.

\par

The following theorem, which is the main result in this section, concerns
Schatten-von Neumann properties for an operator $T_K$
when the operator kernel $K$ belongs to  Gelfand-Shilov spaces.

\par

\begin{thm}\label{thmSchattGSkernels}
Let $\mascB _1$ and $\mascB _2$ be quasi-Banach spaces such that
$$
\maclS _{1/2}(\rr {d_j})\hookrightarrow \mascB _j \hookrightarrow
\maclS _{1/2}'(\rr {d_j}),\quad j=1,2,
$$
and let $p\in (0,\infty ]$. Then the following is true:
\begin{enumerate}
\item if $s\ge 1/2$, $\mascB _1\hookrightarrow \maclS _s'(\rr {d_1})$,
$\maclS _s(\rr {d_2})\hookrightarrow \mascB _2$, and
$K\in \maclS _s(\rr {d_2+d_1})$, then $T_K\in \mascI _p(\mascB _1,\mascB _2)$;

\vrum

\item if $s> 1/2$, $\mascB _1\hookrightarrow \Sigma _s'(\rr {d_1})$,
$\Sigma _s(\rr {d_2})\hookrightarrow \mascB _2$, and
$K\in \Sigma _s(\rr {d_2+d_1})$, then $T_K\in \mascI _p(\mascB _1,\mascB _2)$;

\vrum

\item if $\mascB _1\hookrightarrow \mascS '(\rr {d_1})$,
$\mascS (\rr {d_2})\hookrightarrow \mascB _2$, and
$K\in \mascS (\rr {d_2+d_1})$, then $T_K\in \mascI _p(\mascB _1,\mascB _2)$.
\end{enumerate}
\end{thm}

\par


We need some preparations for the proof. First we note that if
$\mascB _j,\mascC _j$, $j=1,2$, are quasi-Banach spaces and
$T\, :\, \mascB _1\to \mascB _2$, then
\begin{equation}\label{normests}
\nm T{\mascC _1\to \mascC _2} \le C\nm T{\mascB _1\to \mascB _2},
\quad \text{when}\quad \mascC _1 \hookrightarrow \mascB _1\quad
\text{and}\quad  \mascB _2
\hookrightarrow \mascC _2,
\end{equation}
for some constant $C$. Here and in what follows we use the convention that if $T$
is a linear operator from $\mascB _1$ to $\mascB _2$, $\mascC _1\subseteq \mascB _1$
and $\mascB _2 \subseteq \mascC _2$, then the restriction of $T$ to an operator from
$\mascC _1$ to $\mascC _2$ is still denoted by $T$.

\par

\begin{lemma}\label{lemSchattenembs}
Let $\mascB _k,\mascC _k$, $k=1,2$, be quasi-Banach spaces 
such that $\mascC _1 \hookrightarrow \mascB _1$ and $\mascB _2
\hookrightarrow \mascC _2$. Also let $p\in (0,\infty ]$ and $T\, :\, \mascB _1\to
\mascB _2$ be linear and continuous. Then
\begin{equation}\label{singvalineq}
\begin{aligned}
\sigma _j(\mascC _1,\mascC _2,T) &\le C\sigma _j(\mascB _1,\mascB _2,T), \qquad j\ge 1,
\\[1ex]
\text{and}\quad
\nm T{\mascI _p(\mascC _1,\mascC _2)} &\le C\nm T{\mascI _p(\mascB _1,\mascB _2)},
\end{aligned}
\end{equation}
where $C$ is the same constant as in \eqref{normests}.
\end{lemma}

\par

\begin{proof}
It suffices to prove the first inequality in \eqref{singvalineq}. Let
\begin{align*}
\Omega _{j} &= \sets {T_0\in \maclB (\mascB _1,\mascB _2)}{\rank T_0<j},
\\[1ex]
\Omega _{1,j} &= \sets {T_0\in \maclB (\mascC _1,\mascC _2)}{\rank T_0<j},
\end{align*}
and let $\Omega _{2,j}$ be the set of all $T_0$ in $\Omega _{1,j}$ such that $T_0$
is a restriction of an element in $\Omega _{j}$.
Then $\Omega _{2,j}\subseteq \Omega _{1,j}$, and the restrictions of the elements in
$\Omega _j$ to $\mascC _1$ belong to $\Omega _{2,j}$. This gives
\begin{multline*}
\sigma _j(\mascC _1,\mascC _2,T)
=
\inf _{T_0\in \Omega _{1,j}}\nm {T-T_0}{\mascC _1\to \mascC _2}
\le
\inf _{T_0\in \Omega _{2,j}}\nm {T-T_0}{\mascC _1\to \mascC _2}
\\[1ex]
=
\inf _{T_0\in \Omega _{j}}\nm {T-T_0}{\mascC _1\to \mascC _2}
\le
C\inf _{T_0\in \Omega _{j}}\nm {T-T_0}{\mascB _1\to \mascB _2}
=
C\sigma _j(\mascB _1,\mascB _2,T),
\end{multline*}
where the last inequality follows from \eqref{normests}. Hence \eqref{singvalineq}
follows, and the proof is complete.
%
%
\end{proof}
%
%
%
%
%
%
%
%
%
%
%
%
%
%
%

\par

Before stating the next results we need some notions from \cite{To12} concerning Hilbert
spaces $\mascH$, which satisfy
$$
\maclS _{1/2}(\rr d)\hookrightarrow \mascH \hookrightarrow \maclS _{1/2}'(\rr d).
$$
We let
$$
(S_\pi f)(x)\equiv f(x_{\pi (1)},\dots ,x_{\pi (d)})\in \mascH
\quad \text{when}\quad f\in \maclS _{1/2}'(\rr d),
$$
when  $\pi$ is a permutation of $\{ 1,\dots ,d\}$. The Hilbert space $\mascH$
is said to be of Hermite type, if $(h_\alpha /\nm {h_\alpha}{\mascH})_{\alpha}$ is
an orthonormal basis for $\mascH$, and
$\nm {S_\pi f}{\mascH} = \nm f{\mascH}$ for every $f\in \mascH$ and every
permutation $\pi$ on $\{ 1,\dots ,d\}$.

\par

The $L^2$-dual $\mascH '$ of $\mascH$ consists of all $f\in \maclS _{1/2}'(\rr d)$
such that
$$
\nm f{\mascH '}\equiv \sup |(f,\fy ) _{L^2}|
$$
is finite. Here the supremum is taken over all $\fy \in \maclS _{1/2}$ such that
$\nm \fy{\mascH}\le 1$.

\par

We also let $\mascH ^{\tau}$ be the set of all  $f\in \maclS _{1/2}'$ such that
$\overline f\in \mascH$. Then $\mascH$ and $\mascH ^\tau$ with norms
$f\mapsto \nm f{\mascH '}$ and $f\mapsto \nm {\overline f}{\mascH}$ respectively,
are Hilbert spaces.

\par

The following two results are immediate consequences of Propositions 3.8
and 4.9 in \cite{To12}. The proofs are therefore omitted. 

\par

\begin{prop}\label{propA6}
Let $\mascH _j$ be Hilbert space of Hermite types on $\rr {d_j}$ for
$j=1,2$, and let $T$ be a linear and continuous map from $\mascH _1$
to $\mascH _2$. Also let $\mascH =\mascH _2\otimes (\mascH _1')^\tau$
(Hilbert tensor product). If $K_T$ is the kernel of $T$, then
$T\in \mascI _2(\mascH _1,\mascH_2)$, if and only if $K_T\in \mascH$, and
\begin{equation}\label{A14}
\nm T{\mascI _2(\mascH _1,\mascH _2)}=\nm {K_T}{\mascH}.
\end{equation}
\end{prop}

\par

\begin{prop}\label{propA9}
Let $s\ge 1/2$ and let $\mascB$ be a quasi-Banach space such that
$$
\maclS _{1/2}(\rr d)\hookrightarrow \mascB \hookrightarrow \maclS _{1/2}'(\rr d)
$$
holds.
Then the following is true: 
\begin{enumerate}
\item if $\maclS _s(\rr d)\hookrightarrow \mascB$, then there is
a Hilbert space $\mascH$ of Hermite type such that
$\maclS _s(\rr d)\hookrightarrow \mascH
\hookrightarrow \mascB$;

\vrum

\item if $\mascB \hookrightarrow \maclS _s'(\rr d)$, then there is
a Hilbert space $\mascH$ of Hermite type such that
$\mascB \hookrightarrow \mascH \hookrightarrow \maclS _s'(\rr d)$.
%
\end{enumerate}

\par

The same conclusions hold after $s\ge 1/2$, $\maclS _s$ and
$\maclS _s '$ are replaced by $s> 1/2$, $\Sigma _s$
and $\Sigma _s'$ respectively, or after $\maclS _s$ and
$\maclS _s '$ are replaced by $\mascS$ and $\mascS '$
respectively.
\end{prop}

\par

\begin{proof}[Proof of Theorem \ref{thmSchattGSkernels}]
We only prove (1). The other cases follow by similar arguments and are left
for the reader. By Lemma \ref{lemSchattenembs} and Proposition \ref{propA9}
it follows that we may assume that $\mascB _1$ and $\mascB _2$
are Hilbert spaces of Hermite type.

\par

For every integer $N\ge 1$, it follows by repeated application of
Theorem \ref{GSkernels}, that
$$
T_K=T_{K_N}\circ \cdots \circ T_{K_1},
$$
for some $K_j\in \maclS _s$, $j=1,\dots ,N$. Then Proposition \ref{propA6}
shows that $T_{K_j}\in \mascI _2$, for every
$j=1,\dots ,N$. Hence, if $N\ge 2/p$, then H{\"o}lder's inequality for Schatten-von Neumann
operators give
\begin{equation*}
T_K=T_{K_N}\circ \cdots \circ T_{K_1}\in \mascI _2\circ \cdots \circ \mascI _2
\subseteq \mascI _{2/N}\subseteq \mascI _p,
\end{equation*}
which gives the result. The proof is complete.
\end{proof}

\par

\section*{Acknowledgment}
The authors are thankful to Stevan Pilipovi{\'c} for interesting discussions and valuable advice.

\par


\end{document}